\documentclass[11pt]{article}
\usepackage{amssymb,amsmath}
\usepackage{amsmath}
\usepackage{amsfonts}
\newtheorem{precor}{{\bf Corollary}}

\newenvironment{cor}{\begin{precor}{\hspace{-0.5
               em}{\bf.\ }}}{\end{precor}}
\newtheorem{precon}{{\bf Conjecture}}

\newtheorem{predefin}{{\bf Definition}}

\newenvironment{defin}[1]{\begin{predefin}{\hspace{-0.5
                   em}{\bf.\ }}{\rm
#1}\hfill{$\spadesuit$}}{\end{predefin}}
\newtheorem{preexm}{{\bf Example}}

\newtheorem{preappl}{{\bf Application}}

\newtheorem{prelem}{{\bf Lemma}}

\newenvironment{lem}{\begin{prelem}{\hspace{-0.5
               em}{\bf.\ }}}{\end{prelem}}
\newtheorem{preproof}{{\bf Proof.\ }}

\newenvironment{proof}[1]{\begin{preproof}{\rm
               #1}\hfill{$\blacksquare$}}{\end{preproof}}
\newtheorem{presproof}{{\bf Sketch of Proof.\ }}

\newtheorem{prethm}{{\bf Theorem}}

\newenvironment{thm}{\begin{prethm}{\hspace{-0.5
               em}{\bf.\ }}}{\end{prethm}}
\newtheorem{prealphthm}{{\bf Theorem}}

\newenvironment{alphthm}{\begin{prealphthm}{\hspace{-0.5
               em}{\bf.\ }}}{\end{prealphthm}}

\newtheorem{prealphlemma}{{\bf Lemma}}

\newenvironment{alphlemma}{\begin{prealphlemma}{\hspace{-0.5
               em}{\bf.\ }}}{\end{prealphlemma}}
\newtheorem{prepro}{{\bf Proposition}}

\newtheorem{preprb}{{\bf Problem}}

\newtheorem{prequ}{{\bf Question}}

\newenvironment{qu}{\begin{prequ}{\hspace{-0.5
               em}{\bf.\ }}}{\end{prequ}}
\def\conct[#1,#2]{\mbox {${#1} \leftrightarrow {#2}$}}
\def\dconct[#1,#2]{\mbox {${#1} \rightarrow {#2}$}}
\def\deg[#1,#2]{\mbox {$d_{_{#1}}(#2)$}}
\def\mindeg[#1]{\mbox {$\delta_{_{#1}}$}}
\def\maxdeg[#1]{\mbox {$\Delta_{_{#1}}$}}
\def\outdeg[#1,#2]{\mbox {$d_{_{#1}}^{^+}(#2)$}}
\def\minoutdeg[#1]{\mbox {$\delta_{_{#1}}^{^+}$}}
\def\maxoutdeg[#1]{\mbox {$\Delta_{_{#1}}^{^+}$}}
\def\indeg[#1,#2]{\mbox {$d_{_{#1}}^{^-}(#2)$}}
\def\minindeg[#1]{\mbox {$\delta_{_{#1}}^{^-}$}}
\def\maxindeg[#1]{\mbox {$\Delta_{_{#1}}^{^-}$}}
\def\isdef{\mbox {$\ \stackrel{\rm def}{=} \ $}}
\def\dre[#1,#2,#3]{\mbox {${\cal E}_{_{#3}}(#1,#2)$}}
\def\pdre[#1,#2,#3]{\mbox {${\cal P}_{_{#3}}(#1,#2)$}}
\def\var[#1,#2]{\mbox {${\rm Var}_{_{#1}}(#2)$}}
\def\ls[#1]{\mbox {$\xi^{^{#1}}$}}
\def\hom[#1,#2]{\mbox {${\rm Hom}({#1},{#2})$}}
\def\onvhom[#1,#2]{\mbox {${\rm Hom^{v}}(#1,#2)$}}
\def\onehom[#1,#2]{\mbox {${\rm Hom^{e}}(#1,#2)$}}
\def\core[#1]{\mbox {$#1^{^{\bullet}}$}}
\def\cay[#1,#2]{\mbox {${\rm Cay}({#1},{#2})$}}
\def\cays[#1,#2]{\mbox {${\rm Cay_{s}}({#1},{#2})$}}
\def\dirc[#1]{\mbox {$\stackrel{\rightarrow}{C}_{_{#1}}$}}
\def\cycl[#1]{\mbox {${\bf Z}_{_{#1}}$}}

\def\sdg[#1]{\mbox {$\stackrel{\leftrightarrow}{#1}$}}
\begin{document}
%\setcounter{page}{183}
%{\footnotesize AAA {\bf ?} (200?) ?--?}\\
%\maketitle
\footnotetext[1]{This paper is partially supported by Shahid
Beheshti University.}
%\footnotetext[2]{Correspondence should be
%addressed to {\tt hhaji@sbu.ac.ir}.}
\begin{center}
{\Large \bf A Generalization of the Erd\"{o}s-Ko-Rado Theorem}\\
\vspace*{0.5cm}
{\bf Meysam Alishahi, Hossein Hajiabolhassan and Ali Taherkhani}\\
{\it Department of Mathematical Sciences}\\
{\it Shahid Beheshti University, G.C.}\\
{\it P.O. Box {\rm 1983963113}, Tehran, Iran}\\
{\tt m\_alishahi@sbu.ac.ir}\\
{\tt hhaji@sbu.ac.ir}\\
{\tt a\_taherkhani@sbu.ac.ir}\\ \ \\
\end{center}
\begin{abstract}
\noindent In this note, we investigate some properties of local
Kneser graphs defined in \cite{1096284}. In this regard, as a
generalization of the Erd${\rm \ddot{o}}$s-Ko-Rado theorem, we
characterize the maximum independent sets of local Kneser graphs.
Next, we present an upper bound for their chromatic number.
\begin{itemize}
\item[]{{\footnotesize {\bf Key words:}\  local chromatic number, graph homomorphism.}}
%\item[]{ {\footnotesize {\bf Papers title:}  Local Chromatic Number.}}
%\item[]{ {\footnotesize {\bf Abbreviated title:} Local Chromatic Number.}}
\item[]{ {\footnotesize {\bf Subject classification: 05C} .}}
\end{itemize}
\end{abstract}
\section{Introduction}
 In this section we elaborate on some basic definitions and facts
that will be used later. Throughout the paper the word {\it
graph} is used for a finite simple graph with a prescribed set of
vertices.  A {\it homomorphism} $\sigma: G \longrightarrow H$ from
a graph $G$ to a graph $H$ is a map $\sigma: V(G) \longrightarrow
V(H)$ such that $uv \in E(G)$ implies $\sigma(u)\sigma(v) \in
E(H)$. The existence of a homomorphism is indicated by the symbol
$G \longrightarrow H$ (for more on graph homomorphisms
see \cite{MR2089014}).\\

In \cite{MR1067243} Bondy and Hell define $\nu(G,K)$, for two
graphs $G$ and $K$, as the maximum number of vertices in a
subgraph of $G$ that admits a homomorphism to $K$; and using this
they introduce  the following generalization of a result of
Albertson and Collins \cite{MR791653} in which
$\mu(G,K)=|V(G)|/\nu(G,K)$.
\begin{alphthm}{\rm  \cite{MR1067243}}\label{BOHETHM}\
Let $G,H$ and $K$ be graphs where $H$ is a vertex--transitive
graph. If there exists a homomorphism $\sigma: G \longrightarrow
H$ then $\mu(G,K) \leq \mu(H,K)$.
\end{alphthm}

Hereafter, we denote by $[m]$ the set $\{1, 2, \ldots, m\}$, and
denote by ${[m] \choose n}$ the collection of all $n$-subsets of
$[m]$. Suppose $m \geq 2n$ are positive integers. We denote by
$[m]$ the set $\{1, 2, \cdots, m\}$, and denote by ${[m] \choose
n}$ the collection of all $n$-subsets of $[m]$. The {\em Kneser
graph} $KG(m,n)$ has vertex set ${[m] \choose n}$, in which $A
\sim B$ if and only if $A \cap B = \emptyset$. It was conjectured
by Kneser \cite{MR0068536} in 1955 and proved by Lov\'{a}sz
\cite{MR514625} in 1978 that $\chi(KG(m,n))=m-2n+2$.
%%%%%%%%%%%%%%%%%%%%%%%%%%%%%%%%%%%%%%%%%%%%%%%%%%%%%%%%%%%
The {\it local chromatic number} of a graph was defined in
\cite{MR837951} as the minimum number of colors that must appear
within distance $1$ of a vertex. Here is the formal definition.

\begin{defin}{Let $G$ be a graph. Define the local chromatic
number of $G$ as follows
$$\psi(G)\isdef {\displaystyle \min_c}{\displaystyle \max_{v\in V(G)}}|\{c(u): u\in V(G), d_G(u,v)\leq 1\}|,$$
where the minimum is taken over all proper colorings $c$ of $G$
and $d_G(u,v)$ denotes the distance between $u$ and $v$ in $G$.}
\end{defin}

The local chromatic number of graphs has received attention in
recent years \cite{MR2184127, 1096284, MR2279672, MR2452828}.
Clearly, $\psi(G)$ is always bounded from above by the chromatic
number, $\chi(G)$. It is much less obvious that $\psi(G)$ can be
strictly less than $\chi(G)$. In fact, it was proved in
\cite{MR837951}, there exist graphs with $\psi(G)=3$ and
$\chi(G)$ arbitrarily large.

One can define $\psi(G)$ via graph homomorphism. In this regard,
local complete graphs were defined in \cite{MR837951} as follows.

\begin{defin}{
Let $n$ and $r$ be  positive integers where $n\geq r$. Define the
local complete graph $U(n,r)$ as follows.

$$V(U(n,r))=\{(a,A)|\ a\in [n],\ A\subseteq
[n],\ |A|=r-1,\ a\notin A \}$$ and
$$E(U(n,r))=\{\{(a,A),(b,B)\}|\ a\in B,\ b\in A \}.$$

}\end{defin}

The following simple lemma reveals the connection between local
complete graphs and local chromatic number.

\begin{alphlemma}{{\rm(Erd{\H{o}}s et al. \cite{MR837951})}
A graph $G$ admits a proper coloring $c$ with $n$ colors and
$\max_{v\in V(G)}|\{c(u)|\ u\in N [v]\}|\leq r$ if and only if
there exists a homomorphism from $G$ to $U(n,r)$. In particular
$\psi(G)\leq r$ if and only if there exists an $n$ such that $G$
admits a homomorphism to $U(n,r)$.
}\end{alphlemma}
%%%%%%%%%%%%%%%%%%%%%%%%%%%%%%%%%%%%%%%%%%%%%%%%%%%%%%%%%%

In \cite{MR837951} the local complete graphs have been generalized
as follows.
\begin{defin}{\cite{1096284}
Let $n,\ r$ and $t$ be positive integers  where $n\geq r\geq 2t$.
Set $U_t(n,r)$ to be the local Kneser graph whose vertex set
contains all ordered pairs $(A,B)$ such that $|A|=t,\ |B|=r-t,\
A, B\subseteq [n]$ and $A\cap B=\varnothing$. Also, two vertices
$(A,B)$ and $(C,D)$ of $U_t(n,r)$ are adjacent if $ A\subseteq D$
and $C\subseteq B$. }\end{defin}
%%%%%%%%%%%%%%%%%%%%%%%%%%%%%%%%%%%%%%%%%%%%%%%%%%%%%%%%%%%%%%%%%%%%%%
{\bf Remark.} Note that $U_1(n, r) = U(n, r)$, while $U_t(r, r) =
KG(r,t)$. Hence the graph $U_t(n, r)$ provides a common
generalization of Kneser graphs and local complete graphs $U(n,
r)$ in \cite{MR837951}.
%%%%%%%%%%%%%%%%%%%%%%%%%%%%%%%%%%%%%%%%%%%%%%%%%%%%%%%%%%%%%%%%%%%%%%%%

In this paper, we investigate some properties of local Kneser
graphs. In this regard, as a generalization of the Erd${\rm
\ddot{o}}$s-Ko-Rado theorem, we characterize the maximum
independent sets of local Kneser graphs. Next, we provide an
upper bound for their chromatic number.

\section{Local Kneser Graphs}\label{loc}
%%%%%%%%%%%%%%%%%%%%%%%%%%%%%%%%%%%%%%%%%%%%%%%%%%%%%%%%%%%%%%%%%%%%%%%%%%%%%%%%%%%%%%%%%%%%%%%%%%%%%%%%%%%%%%%%%%%%%%%%%

%%%%%%%%%%%%%%%%%%%%%%%%%%%%%%%%%%%%%%%%%%%%%%%%%%%%%%%%%%%%%%%%%%%%%%%%%%%%%%%%%%%%%%%%%%%%%%%%%%%%%%%%%%%%%%%%%%%%%%%%%
In this section we study some properties of the graph $U_t(n,r)$.
First, we characterize the maximum independent sets of $U_t(n,r)$.
To begin we compute the independence number of $U_t(n,r)$. First,
we introduce some notations which will be used throughout the
paper.

Assume that $\sigma$ is a permutation of $[n]$, $R\subseteq [n]$
and $|R|=r$. It should be noted that $\sigma$ provides an ordering
for $[n]$, i.e., $\sigma(1)<\sigma(2)<\cdots<\sigma(n)$. Define
$\min_\sigma R$ to be the minimum member of $R$ according to the
ordering $\sigma$, i.e., $\displaystyle{\min_\sigma R\isdef
\displaystyle{\sigma(\min\{\sigma^{-1}(r)| r \in R\})}}$.

Define
$$V_R\isdef\{(A,B)\ |\  A\cup B=R,\ |A|=t\ and\ A\cap
B=\varnothing\}$$ and  set
$$I_{\sigma,R}\isdef\left\{(A,B)\ |\ \min_\sigma R\in A,\ A\cup B=R,\ |A|=t,\
A\cap B=\varnothing\right \}.$$
Also, define
$$S_{\sigma}\isdef\bigcup_{R\subseteq [n],|R|=r}I_{\sigma,R}.$$\\

Note that the independence number of $U_t(n,r)$ has been computed
in \cite{1096284}. Here we present a proof. It is clear that the
induced subgraph of $U_t(n,r)$ obtained by the vertices in $V_R$
is isomorphic to  the Kneser graph $KG(r,t)$ and it is denoted by
$KG_R(r,t)$. That is why we call the graph $U_t(n,r)$ the local
Kneser graph. It is straightforward to check that for every
$\sigma\in S_n $, $I_{\sigma,R}$ is a maximum independent set of
$KG_R(r,t)$. Also, one can  easily see that $S_{\sigma}$ is an
independent set in $U_t(n,r)$ of order ${r-1 \choose t-1}{n
\choose r}$. Clearly, $KG(r,t)$ is a subgraph of $U_t(n,r)$.
Hence, $KG(r,t)\rightarrow U_t(n,r)$. By using Bondy and Hell
theorem \cite{MR1067243}, we have
$$\frac{{r \choose t}}{{{r-1} \choose {t-1}}}  \leq \frac{{r
\choose t}{ n\choose r}}{\alpha (U_t(n,r))}.$$ Hence, $\alpha
(U_t(n,r))\leq{r-1 \choose t-1}{n \choose r}$. Consequently,
$\alpha (U_t(n,r))={r-1 \choose t-1}{n \choose r}$ and
$S_{\sigma}$ is a maximum
independent set of $U_t(n,r)$ .\\
Now, we are ready to show that for every maximum independent set
$S$ in $U_t(n,r)$ there exists  a permutation $\sigma$ of $S_n$
such that $S=S_\sigma$.

Consider a  maximum independent set $S$ in $U_t(n,r)$. Note that
$|S|={r-1\choose t-1}{n\choose r}$. One can easily see that  for
every $R\subseteq [n]$ ($|R|=r$), $V_R\cap S$ is a maximum
independent set in $KG_R(r,t)$. By the ~Erd{\H{o}}s-Ko-Rado
theorem \cite{MR0140419}, there is an $x(S,R)\in R$ such that
$x(S,R)\in \displaystyle{\bigcap_{(A,B)\in V_R\cap S}}A$.

\begin{lem}
\label{lem} { Let $S$ be a maximum independent set in $U_t(n,r)$
where $n\geq r>2t$. Also, assume that  $R,\ R'$ are two distinct
$r$-subsets of $[n]$. If $x(S,R)=x\in R\cap R'$, then
$x(S,R')\notin R\cap R'\setminus \{x\}$.

}\end{lem}
\begin{proof}{
Assume that $x(S,R)=x$ and $x(S,R')=z$. We prove this lemma by
induction on $|R\setminus R'|$.\\
Let $|R\setminus R'|=1 $. Then there are $u\in R$ and  $v\in R'$
such that $R=(R'\setminus \{v\})\cup \{u\}$. If $x(S,R')=z\in
R\cap R'\setminus \{x\}$, then  there exist $(A,B)\in S$ and
$(A',B')\in S$ such that $A,B\subset R$, $x\in A$,  $u,z\in B$ and
 $A',B'\subset R'$, $z\in A'$, $x,v\in B'$,
$A'\subset B$, $A\subset B'$. Hence, $(A,B)$ and $(A',B')$ are
adjacent which is a contradiction.

Suppose that $k>1$ and the lemma holds for $|R\setminus R'|<k$.
Now, let $|R\setminus R'|=k$. On the contrary, assume that $z\in R
\cap R'$. Choose $y\in R\setminus R'$ and $y'\in R'\setminus R$
and set $R''=(R'\setminus \{y'\})\cup \{y\}$. Since $|R'\setminus
R''|=1<k$, we have $x(S,R'')\notin R'\cap R''\setminus \{z\}$;
consequently, $x(S,R'')\in\{y,z\}$. On the other hand,
$|R\setminus R''|=k-1<k$; hence, $x(S,R'')\notin R\cap
R''\setminus \{x\}$. But, $\{y,z\}\subset R\cap R''\setminus
\{x\}$ which is a contradiction.}\end{proof}

Now, we characterize the maximum independent sets of local Kneser
graphs.

\begin{thm}
Let $S$ be a maximum independent set in $U_t(n,r)$. Then there
exists a permutation  $\sigma\in S_{n}$ such that $S=S_{\sigma}$.
\end{thm}
\begin{proof}{
Suppose $S$ is a maximum independent set in $U_t(n,r)$. We define
a directed graph $D_{S}$ whose vertex set and edge set are
$$V(D_{S})=\{1,2,\ldots ,n\}$$ and \\
%$$E(D_{S})=\{(i,j)\mid\, \exists\,(A,B)\in S\,\, s.t. \,\, i=x(S,A\cup B)\ and \ j\in A\cup B \}.$$
\begin{centerline}{$E(D_{S})\isdef\{(i,j)\mid\, \exists\,R\subseteq [n],
|R|=r,\,\,i\neq j,\, \{i,j\}\subseteq R,\, i=x(S,R) \},$
respectively.}
\end{centerline}

Assume that ${d}_1\geq {d}_2\geq\cdots \geq {d}_n$ is the out
degree sequence of $D_{S}$ where ${d}_{i}$ is the out degree of
$v_{i}$ for $i=1,2,\ldots,n$ . In view of Lemma \ref{lem}, one
can see that $D_{S}$ is a directed graph with no multiplicity.
Consequently, $|S|= {d_{1}\choose r-1}{r-1 \choose
t-1}+{d_{2}\choose r-1}{r-1 \choose t-1}+\cdots+{d_{n}\choose
r-1}{r-1 \choose t-1}$. However, ${d_{1}\choose
r-1}+{d_{2}\choose r-1}+\cdots+{d_{n}\choose r-1}$ is maximized
when $d_{1}=n-1,\,d_{2}=n-2,\,\ldots,d_{n-r+1}=r-1$. Choose a
permutation  $\sigma\in S_{n}$ such that $\sigma(i)=v_i $ for
$i=1,2,\ldots,n-r+1$. Obviously, $S=S_\sigma$. }\end{proof}

From the above discussion, directed graph $D_S$ is related to the
independent set $S$ of $U_t(n,r)$. Conversely, suppose that $D$
is a directed graph on $[n]$ with no multiplicity. Now, we want
to construct an independent set $I_D$ which is related to $D$. Set
$$I_{D}=\{(A,B)\,|\,\exists\ i\in [n];  \, A, B\subseteq N^+(i)\cup \{i\},\ i\in
A,\ A\cap B=\varnothing,\ |A|=t \, \ |B|=r-t \},$$ where
$N^+(i)=\{j\ |\ (i,j)\in E(D)\}$. Clearly, $I_{D}$ is an
independent set in $U_t(n,r)$. It is easy to see that for any
maximum independent set $S$ in $U_t(n,r)$ we have $I_{D_S}=S$.

%%%%%%%%%%%%%%%%%%%%%%%%%%%%%%%%%%%%%%%%%%%%%%%%%%%%%%%%%%%%%%%%%%%%%%%%%%%%
Clearly, $\sigma:U_t(n,r)\longrightarrow KG(n,t)$ is a
homomorphism where $\sigma((A,B))\isdef A$. Therefore,
$\chi(U_t(n,r))\leq n-2t+2$. The chromatic number of local
complete graphs has been investigated in \cite{MR837951}.

\begin{alphthm}{\rm \cite{MR837951}}\label{bb}
Let $n$ and $r$ be positive integers  where $n\geq r$. We have
$\chi(U(n, r)) \leq r 2^r \log_2\log_2 n$.
\end{alphthm}

Here we introduce an upper bound for the chromatic number of local
Kneser graphs.

%%%%%%%%%%%%%%%%%%%%%%%%%%%%%%%%%%%%%%%%%%%%%%%%%%%%%%%%%%%%%%%%%%
\begin{thm}\label{lov}
If $n,\ r$ and $t$ are positive integers  where $n\geq r\geq 2t$,
then $\chi(U_t(n,r))\leq\lceil\frac{r^2}{t}(\ln n +1)\rceil$
\end{thm}
\begin{proof}{
Assume that  $\sigma_{1},\sigma_{2},\ldots,\sigma_{l}$ are $l$
random permutations of $S_n$ such  that they have been chosen
independently and uniformly. For each vertex $(A,B)\in
V(U_t(n,r))$, define ${\cal E }_{(A,B)}$ to be the event that
$(A,B)\notin \bigcup S_{\sigma_i}$. Obviously, $(A,B)\in S_\sigma$
if and only if  there exists $a\in A$ such that $a$  precedes all
 elements of $A\cup B\setminus\{a\}$ in $\sigma$. Clearly,
$Pr({\cal E }_{(A,B)})=(1-\frac{t}{r})^l$. Consider a random
variable $X$ where $X(\sigma_1,\ldots ,\sigma_l)\isdef
|V(U_t(n,r))\setminus \bigcup S_{\sigma_i}|$. Clearly,
$E(X)={r\choose t }{n\choose r}(1-\frac{t}{r})^l$. If
$l=\lceil\frac{r^2}{t}(\ln(en)\rceil$, then $E(X)<1$. Hence,
$\chi(U_t(n,r))\leq\lceil\frac{r^2}{t}(\ln(en)\rceil$.}
\end{proof}
Theorem \ref{lov} immediately yields the following corollary.
\begin{cor}\label{lov1}
Let $n$ and $r$ be positive integers  where $n\geq r$. We have
$\chi(U(n, r)) \leq \lceil r^2(\ln n +1)\rceil$.
\end{cor}
In other word, previous corollary says; if we have a proper
coloring for graph $G$ with $n$ colors which assigns at most $r$
colors in the closed neighborhood of every vertex, then
$\chi(G)\leq \lceil r^2(\ln n +1)\rceil$. Two upper bounds in
Theorem \ref{bb} and Corollary \ref{lov1} are complementary.

Note that $KG(r,t)$ is a subgraph of $U_t(m,r)$; consequently,
$r-2t+2$ is a lower bound for the chromatic number of $U_t(m,r)$
while here we show that $r-2t+2$ is an upper bound for the local
chromatic number of $U_t(m,r)$.

\begin{lem}
Assume that $n,\ r$ and $t$ are positive integers  where $n\geq
r\geq 2t$. Then $\psi(U_t(m,r))\leq r-2t+2$.
\end{lem}
\begin{proof}{
Let $(A,B)\in V(U_t(m,r))$ , $A=\{a_1,a_2,\ldots,a_t\}$ and
$B=\{b_1,b_2,\ldots,b_{r-t}\}$ such that $a_1<a_2<\cdots<a_t$ and
$b_1<b_2<\cdots<b_{r-t}$. Now, we show that there exists a graph
homomorphism from $U_t(m,r)$ to $U(m-t+1,r-2t+2)$. To see this,
define $f((A,B))\isdef (\min A,B^*)$ where $\min A=a_1$ and
$B^*\isdef \{b_1,b_2,\ldots,b_{r-2t+1}\}$. If $(A,B)$ and $ (C,D)$
are adjacent in $U_t(m,r)$, then obviously $\min A\in D^*$, $\min
C \in B^*$ and $\min A\neq \min C$. Therefore, $f$ is a graph
homomorphism, as desired.}
\end{proof}
%%%%%%%%%%%%%%%%%%%%%%%%%%%%%%%%%%%%%%%%%%%%%%%%%%%%%%%%%
The aforementioned lemma motivates us to propose the following
question.

\begin{qu}
Assume that $n,\ r$ and $t$ are positive integers  where $n\geq
r\geq 2t$. Is it true that $\psi(U_t(m,r))=r-2t+2$?
\end{qu}
%%%%%%%%%%%%%%%%%%%%%%%%%%%%%%%%%%%%%%%%%%%%%%%%%%%%%%%%%%%%%%%%%%%
%%%%%%%%%%%%%%%%%%%%%%%%%%%%%%%%%%%%%%%%%%%%%%%%%%%%%%%%%%%%%%%%%%%
%%%%%%%%%%%%%%%%%%%%%%%%%%%%%%%%%%%%%%%%%%%%%%%%%%%%%%%%%%%%%%%%%%%%%%%%%%%%%%%%%%%%%%%%
%%%%%%%%%%%%%%%%%%%%%%%%%%%%%%%%%%%%%%%%%%%%%%%%%%%%%%%%%%%%%%%%%%%%%%%%%%%%%%%%%%%%%%%%


\begin{thebibliography}{10}
\bibitem{MR791653}
Michael~O. Albertson and Karen~L. Collins.
\newblock Homomorphisms of {$3$}-chromatic graphs.
\newblock {\em Discrete Math.}, 54(2):127--132, 1985.

\bibitem{MR2184127}
Ivo Bl{\"o}chliger and Dominique de~Werra.
\newblock Locally restricted colorings.
\newblock {\em Discrete Appl. Math.}, 154(1):158--165, 2006.

\bibitem{MR1067243}
J.~A. Bondy and Pavol Hell.
\newblock A note on the star chromatic number.
\newblock {\em J. Graph Theory}, 14(4):479--482, 1990.

\bibitem{MR837951}
P.~Erd{\H{o}}s, Z.~F{\"u}redi, A.~Hajnal, P.~Komj{\'a}th,
V.~R{\"o}dl, and
  {\'A}.~Seress.
\newblock Coloring graphs with locally few colors.
\newblock {\em Discrete Math.}, 59(1-2):21--34, 1986.

\bibitem{MR0140419}
P.~Erd{\H{o}}s, Chao Ko, and R.~Rado.
\newblock Intersection theorems for systems of finite sets.
\newblock {\em Quart. J. Math. Oxford Ser. (2)}, 12:313--320, 1961.

\bibitem{MR2089014}
Pavol Hell and Jaroslav Ne{\v{s}}et{\v{r}}il.
\newblock {\em Graphs and homomorphisms}, volume~28 of {\em Oxford Lecture
  Series in Mathematics and its Applications}.
\newblock Oxford University Press, Oxford, 2004.

\bibitem{MR0068536}
Martin Kneser.
\newblock Aufgabe 360.
\newblock {\em Jahresbericht der Deutschen Mathematiker-Vereinigung}, 58:27,
  1955.

\bibitem{1096284}
J\'{a}nos K\"{o}rner, Concetta Pilotto, and G\'{a}bor Simonyi.
\newblock Local chromatic number and sperner capacity.
\newblock {\em J. Comb. Theory Ser. B}, 95(1):101--117, 2005.

\bibitem{MR514625}
L.~Lov{\'a}sz.
\newblock Kneser's conjecture, chromatic number, and homotopy.
\newblock {\em J. Combin. Theory Ser. A}, 25(3):319--324, 1978.

\bibitem{MR2279672}
G{\'a}bor Simonyi and G{\'a}bor Tardos.
\newblock Local chromatic number, {K}y {F}an's theorem and circular colorings.
\newblock {\em Combinatorica}, 26(5):587--626, 2006.

\bibitem{MR2452828}
G{\'a}bor Simonyi, G{\'a}bor Tardos, and Sini{\v{s}}a~T.
Vre{\'c}ica.
\newblock Local chromatic number and distinguishing the strength of topological
  obstructions.
\newblock {\em Trans. Amer. Math. Soc.}, 361(2):889--908, 2009.

\end{thebibliography}
\end{document}